\documentclass[12pt,a4paper]{mathscan}
\usepackage{amsfonts,amssymb}
\usepackage[english]{babel}
\usepackage{enumerate}
\usepackage{pb-diagram}
\usepackage{color}

\title
{Quasi-multipliers of Hilbert and Banach $C^*$-bimodules}

\author{Alexander Pavlov}
\thanks{
 }
\address{All-Russian Institute of Scientific and Technical Information,
Russian Academy of Sciences (VINITI RAS), Usievicha str. 20,
125190 Moscow A-190, Russia}
\email{axpavlov@gmail.com}

\author{Ulrich Pennig}
\address{Mathematisches Institut, Georg-August-Universit\"{a}t, Bunsenstra{\ss}e
3-5,
D-37073 G\"{o}t\-tin\-gen, Germany}
\email{pennig@math.uni-goettingen.de}

\author{Thomas Schick}
\address{Mathematisches Institut, Georg-August-Universit\"{a}t, Bunsenstra{\ss}e
3-5,
D-37073 G\"{o}t\-tin\-gen, Germany}
\email{schick@uni-math.gwdg.de}

\newtheorem{teo}{Theorem}[section]
\newtheorem{lem}[teo]{Lemma}
\newtheorem{prop}[teo]{Proposition}

\newtheorem{cor}[teo]{Corollary}

\newtheorem{dfn}[teo]{Definition}
\newtheorem{rk}[teo]{Remark}
\theoremstyle{remark}
\newtheorem{ex}[teo]{Example}

\begin{document}
\begin{abstract}
Quasi-multipliers for a Hilbert $C^*$-bimodule $V$ were introduced
by L. G. Brown, J. A. Mingo, and N.-T. Shen (Canad. J. Math., {\bf 46}(1994), 1150-1174) as a certain subset of the Banach bidual
module $V^{**}$. We give another (equivalent) definition of
quasi-multipliers for Hilbert $C^*$-bimodules using the
centralizer approach and then show that quasi-multipliers are, in
fact, universal (maximal) objects of a certain category. We also
introduce quasi-multipliers for bimodules in Kasparov's sense and
even for Banach bimodules over $C^*$-algebras, provided these
$C^*$-algebras act non-degenerately. A topological picture of
quasi-multipliers via the quasi-strict topology is given. Finally,
we describe quasi-multipliers in two main situations: for the
standard Hilbert bimodule $l_2(A)$ and for bimodules of
sections of Hilbert $C^*$-bimodule bundles over locally compact
spaces.
\end{abstract}

\def\Hom{{\rm Hom}}
\def\End{{\rm End}}
\def\vHom{\mathbb{HOM}}
\def\vEnd{\mathbb{END}}

 \maketitle

%%%%%%%%%%%%%%%%%%%%%%%%%%%%%%%
\section{Introduction}
%%%%%%%%%%%%%%%%%%%%%%%%%%%%%%%
There are several equivalent ways to introduce quasi-multipliers
(left as well as right and (double) multipliers) for a
$C^*$-algebra $A$. It may be done in terms of centralizers
(cf.~\cite{Busby}), via universal representations treating $A$ as
a $C^*$-subalgebra of its enveloping von Neumann algebra $A^{**}$
(cf., e.g., \cite[\S 3.12]{Pedersen}) and by a categorical
approach describing multipliers as universal objects in suitable
categories (\cite[Ch.2]{Lance}, \cite{Pav1}). These theories were
extended to the category of Hilbert $C^*$-(bi)modules. More
precisely, in
  this context multipliers
were defined and studied in \cite{Bakic, RT02}, left multipliers
in \cite{FraPav07} and quasi-multipliers in \cite{BrMinShen}.
These concepts coincide with the theories for $C^*$-algebras in
the particular situation when the Hilbert (bi)module under
consideration is nothing else but the underlying $C^*$-algebra.

Our aim in this work is to define and study quasi-multipliers for
Hilbert $C^*$-bimodules, Hilbert bimodules in Kasparov's sense
and, more generally, even for Banach bimodules over
$C^*$-algebras, on which both algebras act non-degenerately. For
Hilbert $C^*$-bimodules our definition of quasi-multipliers
differs from the one of \cite{BrMinShen}, but, as we show, these
definitions are actually equivalent. We introduce
quasi-multipliers using the centralizer approach, and then show
that these objects are, in fact, universal (maximal) objects of
some categories. Note that in \cite{BrMinShen} quasi-multipliers
of a Hilbert $C^*$-bimodule $V$ are considered as a certain subset
of the Banach bidual module $V^{**}$ that allows to characterize
embeddings of Hilbert $C^*$-bimodules into $C^*$-algebras,
\cite[Theorem 4.3]{BrMinShen}. We study also the quasi-strict
topology and give the topological picture of quasi-multipliers in
terms of this topology.

Finally, we give the description for quasi-multipliers in two main
situations: for standard bimodules $l_2(A)$ (actually, we obtain a
much more general result concerning quasi-multipliers of infinite
direct sums of bimodules) and for the "commutative" case. The latter
  means that,
for a given locally compact space $X$ and a Hilbert $A$-$B$
bimodule $V$, we treat quasi-multipliers of the Hilbert
$A_0(X)$-$B_0(X)$-bimodule $\mathcal{V}_0(X) = C_0(X,
\mathcal{V})$. These are the continuous sections of a Hilbert $A$-$B$-bimodule
bundle $\mathcal{V}$ over $X$ with typical fiber $V$. Moreover,
$A_0(X)$ and $B_0(X)$ denote the set of continuous $A$-valued and
$B$-valued functions on $X$ vanishing at infinity.

%%%%%%%%%%%%%%%%%%%%%%%%%%%%%%%%%%%%%%%%%%%%%%%%%%%%
\section{Quasi-multipliers of Hilbert $C^*$-bimodules}
%%%%%%%%%%%%%%%%%%%%%%%%%%%%%%%%%%%%%%%%%%%%%%%%%%%%
Given a $C^*$-algebra $A$ and a Banach space  $Q$, recall that $Q$
is said to be an \emph{involutive Banach space} if it is equipped
with a sesqui-linear involution $*\colon Q\rightarrow Q$ such that
$\|q^*\|=\|q\|$ for any $q\in Q$. We will also need some
definitions of \cite{Pav1}.

\begin{dfn}{\rm An involutive Banach space $Q$ with involution $q \mapsto q^*$ is called an \emph{$A$-bimodule} if
 there is a map, which is conjugate linear in the first variable
 and linear in the second variable
\[
 A\times Q\rightarrow Q,\quad (a,q)\mapsto a {\lhd} q,
\]
and a bilinear map
 \[
      Q\times A\rightarrow Q,\quad (q,a)\mapsto q\rhd a
\]
 such that
 \begin{gather*}
    (ba)\lhd q=a\lhd(b\lhd q),\quad q\rhd(ab)=(q\rhd a)\rhd b,\\
    (a\lhd q)\rhd b= a\lhd (q\rhd b),\\
     (a\lhd q\rhd b)^*= b\lhd q^*\rhd a,\\
     \|a\lhd q\|\le\|a\|\|q\|,\quad \|q\rhd b\|\le\|q\|\|b\|
      \end{gather*}
 for all $a, b\in A$, $q\in Q$.
 }
  \end{dfn}

  \begin{dfn}{\rm
Let $Q$ be a bimodule over $A$. Moreover assume that $A\subset Q$
 is an involutive Banach subspace. $A$ is said to be a \emph{quasi-ideal}
 of $Q$ if
 \[
 a\lhd b=a^*b,\quad b\rhd a=ba \quad\mbox{for } a,b\in A
 \]
 and $A\lhd q\rhd A\subset A$ for any $q\in Q$.
 }
  \end{dfn}

  \begin{prop}(\cite[comments to Definition 3]{Pav1})
    Let $A\subset Q$ be a quasi-ideal and $Q^{(0)}=\{q\in Q : A\lhd q\rhd A=0\}$.
   Then $Q^{(0)}$ is a sub-bimodule of $Q$ and the following conditions are equivalent.
   \begin{enumerate}[{\rm (i)}]
     \item $Q^{(0)}=0$;
     \item For any $A$-sub-bimodule $X$ of $Q$ the condition
     $X\cap A=\{0\}$ implies $X=\{0\}$.
   \end{enumerate}
 \end{prop}

   \begin{dfn}{\rm A quasi-ideal $A\subset Q$ is \emph{essential} if it
  satisfies one of the equivalent conditions above.
  }
 \end{dfn}

    \begin{dfn}{\rm
     A quasi-ideal $A\subset Q$ is \emph{strictly essential} if
 \[
 \sup\{\|a\lhd q\rhd b\| : a,b\in A, \|a\|\le1, \|b\|\le 1\}=\|q\|
 \]
 for all $q\in Q$.
 }
 \end{dfn}

   Quasi-multipliers $QM(A)$ of $A$ may be, actually, introduced in
  several equivalent ways, but we prefer here to use their original
  definition in terms of quasi-centralizers (cf. \cite{Busby}).

\begin{dfn}{\rm
A \emph{quasi-multiplier} of $A$ is a bilinear bounded map
$q\colon  A\times A\rightarrow A$ such that
\[
q(ca,bd)=cq(a,b)d\quad\text{for } a,b,c,d\in A.
\]
}
\end{dfn}

The set of quasi-multipliers $QM(A)$ is an involutive Banach space
with respect to the operator norm $\|q\|:=\sup\{\|q(a,b)\| :
\|a\|\le
  1,\,\|b\|\le 1\}$ and the involution:
$q^*(a,b)=q(b^*,a^*)^*$, where $a, b\in A$, $q\in QM(A)$ (cf.
\cite[3.12.2]{Pedersen}).

\begin{prop} (\cite{Pav1})
$A$ is embedded into $QM(A)$ as an involutive Banach subspace via
the $*$-inclusion
\[
a\mapsto q_a,\quad q_a(b,c)=bac,
\]
$a, b, c\in A$. Moreover, $A$ is actually a strictly essential
quasi-ideal of $QM(A)$ and $QM(A)$ is maximal (with respect to
injective homomorphisms of involutive Banach spaces acting
identically on $A$) among all quasi strictly essential extensions
of $A$.
\end{prop}

Now we are going to adopt the considerations of \cite{Bakic,
FraPav07} about double and left  multipliers of Hilbert
$C^*$-modules to introduce quasi-multipliers in the $C^*$-module
context. But, as we saw before, even for $C^*$-algebras we need a
bimodule structure for the definition of quasi-multipliers.
Consequently, we need Hilbert $C^*$-bimodules (moreover, equipped
with some involution) instead of usual Hilbert $C^*$-modules for
the following considerations. Thus, we come to the following
definition.

\begin{dfn}{\rm
A \emph{Hilbert $A$-$B$-bimodule} $V$ is both: a left Hilbert $A$-module
and a right Hilbert $B$-module with commuting actions such that its left
$_A\langle\cdot,\cdot\rangle$ and right $\langle\cdot,\cdot\rangle_B$
inner products satisfy the condition
\[
_A\langle x,y\rangle\,z=x\,\langle y,z\rangle_B
\]
for all $x, y, z\in V$. If $V$ is a Hilbert $A$-$A$-bimodule and a
Banach involutive space such that
\[
(ax)^*=x^*a^*,\quad (xa)^*=a^*x^*,\quad {for }\; a\in A, x\in V.
\]
is said to be an \emph{involutive Hilbert $A$-bimodule}. }
\end{dfn}

The two norms defined on $V$,
one from each inner product necessarily
coincide by \cite[Corollary 1.11]{BrMinShen}.

\begin{ex}
Any $C^*$-algebra may be considered as an involutive Hilbert
bimodule over itself with respect to the inner products $_A\langle
a,b\rangle=ab^*$ and $\langle a,b\rangle_A=a^*b$, where $a, b\in
A$. Obviously, any free module $A^n$ is an involutive Hilbert
bimodule. Observe, however, that the standard module $l^2(A)$ in
  general is not involutive, as was pointed out to us by the referee.
 \end{ex}

\begin{ex}
Any right Hilbert $A$-module $V$ may be considered as a Hilbert
$K(V)$-$A$-bimodule with respect to the inner product
\[
 _{K(V)}\langle x,y \rangle = x\,\langle y, \cdot \rangle_A\ .
\]
\end{ex}

 \begin{ex}
Let $A$ be a $C^*$-subalgebra of a $C^*$-algebra $B$ and
 $E\colon B\rightarrow A$ be a \emph{conditional expectation}, i.e. a surjective projection
 of norm one satisfying  the following conditions:
 \[
E(ab)=aE(b),\quad E(ba)=E(b)a,\quad E(a)=a,
 \]
 for $a\in A$, $b\in B$ (cf.~\cite{Wata}). Then $B$ (with its $C^*$-algebra
 involution) is an involutive
  pre-Hilbert $A$-bimodule with respect to the inner products
 $_A\langle x,y\rangle=E(xy^*)$ and $\langle x,y\rangle_A=E(x^*y)$.
 This module is Hilbert if and only if $E$ is \emph{topologically of index-finite type},
 i.e. the mapping $(K\cdot E-{\rm id}_B)$ is positive for some real number $K\ge 1$
 (cf. \cite{FrankKirJOP, FMTZAA}).
 \end{ex}

\begin{dfn}{\rm
Given two $C^*$-algebras $A$ and $B$ and a Hilbert
$A$-$B$-bimodule $V$, the quasi-multipliers of $V$ are defined as
the set of all bounded $A$-$B$-bilinear homo\-mor\-phisms from
$A\times B$ to $V$,
\begin{equation}
QM(V)={\rm Hom}_{A,B}(A\times B, V)\label{eq:defQM},
\end{equation}
with norm $\|q\|:=\sup\{\|q(a,b)\|\mid a\in A,b\in B \text{ with }\|a\|\le
1,\|b\|\le 1\}$. }
\end{dfn}

$V$ is isometrically embedded into $QM(V)$ by the map
\begin{gather}\label{eq:V_into_QM(V)}
\Gamma \colon  V\rightarrow QM(V),\quad  \Gamma(x)(a,b)=axb,
\end{gather}
and we will identify $V$ with its image under this embedding. If
$V$ is an involutive Hilbert $A$-$A$-bimodule, then $QM(V)$
carries an involution $T^*(a,b)=T(b^*,a^*)^*$ with respect to
which quasi-multipliers $QM(V)$ form an involutive Banach space.

\begin{rk}{\rm
In \cite{BrMinShen} quasi-multipliers were defined via the bidual
$V^{**}$ of $V$ as a Banach space by the formula
\[
 \widetilde{QM}(V) = \{t \in V^{**}\ |\ atb \in V \text{ for all } a\in A, b\in B\}\ .
\]
This definition actually coincides with the one above in
the following sense. Clearly, every element $t \in \widetilde{QM}(V)$
defines a bimodule homomorphism
\[
 q_t \colon  A \times B \longrightarrow V,  \quad (a,b) \mapsto atb.
\]
That means there is a linear map
\[
\varphi \colon  \widetilde{QM}(V)\rightarrow QM(V),\quad t\mapsto q_t,
\]
which, in fact, is an isometry, because
\[
\|q_t\|=\sup\{\|atb\| : \|a\|\le 1, \|b\|\le 1\}=\|t\|
\]
for any $t\in \widetilde{QM}(V)$ by \cite[Lemma 4.1
(iii)]{BrMinShen}.  To see that $\varphi$ is surjective, let $q
\in QM(V)$ be given, choose approximate  units $\{e_{\alpha}\}$ in
$A$ and $\{u_{\beta}\}$ in $B$. Then by \cite[Lemma 4.1
(iv)]{BrMinShen} there is $t\in \widetilde{QM}(V)$ such that
$q(a,b)=\lim_{\alpha, \beta} aq(e_{\alpha},u_{\beta})b=atb$. Such
$t$ is just a $\sigma(V^{**},V^{*})$ cluster point of the bounded
net $\{q(e_{\alpha},u_{\beta})\}$, which has to exist by the
Banach-Alaoglu theorem}.
\end{rk}

\begin{dfn}{\rm
Given two Banach algebras $\mathcal{A}$ and $\mathcal{B}$,
a Banach space $W$ is called a
\emph{Banach-$\mathcal{A}$-$\mathcal{B}$-bimodule}
if it is equipped with a norm continuous left action of $\mathcal{A}$ and a
norm continuous right action of $\mathcal{B}$, such that both actions commute.}
\end{dfn}

\begin{dfn}{\rm
Let $V$ be a Hilbert $A$-$B$-bimodule. The \emph{left multipliers} of $V$ are
\[
 LM(V) = \text{Hom}_{-,B}(B,V) \ ,
\]
i.e. the $B$-linear homomorphisms from $B$ to $V$. The corresponding \emph{right multipliers} are given by
\[
 RM(V) = \text{Hom}_{A,-}(A,V)\ .
\]}
\end{dfn}
In particular $LM(A)$, where $A$ is considered as an
$A$-$A$-bimodule is a Banach algebra with multiplication given by
composition of homomorphisms. In a similar way, we turn $RM(A)$
into a Banach algebra, but here we will use the \emph{opposite
multiplication}, i.e.
\[
 \alpha_1 \cdot \alpha_2 := \alpha_2 \circ \alpha_1
\]
for $\alpha_i \in RM(A)$. With this convention $A$ is a left ideal in $LM(A)$
and a right ideal in $RM(A)$.

Define $QM(V):=\Hom_{A,B}(A\times B,V)$ as the set of bounded
$(A,B)$-bilinear maps as in \eqref{eq:defQM}.

$QM(V)$ comes equipped with an $A$-$B$-bimodule structure in the following way.
Let $a,a' \in A, b,b' \in B, q \in QM(V)$, then
\[
  (q \rhd b)(a',b') := q(a', b\,b') \quad , \quad (a \lhd q)(a',b') = q(a'a, b') \ .
\]
This can be extended to a Banach $RM(A)$-$LM(B)$-bimodule structure via
\[
 (q \rhd \beta)(a,b) := q(a, \beta(b)) \quad , \quad (\alpha \lhd q)(a,b) = q(\alpha(a), b)\ .
\]
for $\alpha \in RM(A)$, $\beta \in LM(B)$.

\begin{rk}{\rm
Obviously, if $A$ is unital, then $QM(V)=LM(V)$. If $B$ is unital,
then $QM(V)=RM(V)$. And if both $A$ and $B$ are unital, then
$QM(V)=V$.}
\end{rk}

Define a locally convex \emph{quasi-strict topology} (we will
denote it by the abbreviation $q.s.$) on ${\rm Hom}_{A,B}(A\times
B, V)$ by the family of semi-norms
\[
\{\nu_{a,b} : a\in A, b\in B\}, \text{ where }
\nu_{a,b}(q)=\|a\lhd q\rhd b\|,\quad q\in {\rm Hom}_{A,B}(A\times
B, V)
\]
and define $X:=[V]_{q.s.}$ as the completion of $V$ with respect
to the quasi-strict topology, restricted to $V$. Now consider a
Cauchy net $\mathbf{x}=\{x_i\}$ in the topological space $(V,
q.s.)$. For any $a\in A$, $b\in B$ the net $\{ax_ib\}$ converges
to some vector $q_{\mathbf{x}}(a,b)\in V$.

\begin{prop}\label{prop:quasi-strict_completion}
The correspondence $\mathbf{x}\mapsto q_{\mathbf{x}}$
 is a linear isometric  map from $X$ onto $QM(V)$. In
the other words, quasi-multipliers of $V$ coincide with the
completion of $V$ with respect to the quasi-strict topology.
\end{prop}

\begin{proof} Obviously,  $q_{\mathbf{x}}$ is a bilinear map for
any Cauchy net $\mathbf{x}=\{x_i\}$ of the space $(V, q.s.)$.  By
the Banach-Steinhaus theorem the set of real numbers $\{\|x_i\|\}$
is bounded, say by a constant $C$. Then
$\|q_{\mathbf{x}}(a,b)\|\le C\|a\|\|b\|$, so $q_{\mathbf{x}}$
actually belongs to $QM(V)$. Now
 let $q \in QM(V)$ be given, choose approximate  units $\{e_{\alpha}\}$ in
$A$ and $\{u_{\beta}\}$ in $B$. Since
\[
 (a\lhd q\rhd b) (e_{\alpha},u_{\beta})\ = q(e_{\alpha}a, b u_{\beta}) \rightarrow q(a,b)
\]
for all $a\in A, b\in B$, the net
$\mathbf{y}=\{q(e_{\alpha},u_{\beta})\}$ is a Cauchy net in $(V,
q.s.)$ and $q=q_{\mathbf{y}}$, so $X=QM(V)$ as required.
\end{proof}

Consider also the locally convex \emph{strong topology} (we will
denote it by the abbreviation $s$) of point-wise convergence
on ${\rm Hom}_{A,B}(A\times B, V)$ defined by the family of
semi-norms
\[
\{\mu_{a,b} : a\in A, b\in B\}, \text{ where } \mu_{a,b}(q)=\|
q(a,b)\|,\quad q\in {\rm Hom}_{A,B}(A\times B, V).
\]
Both these topologies ---quasi-strict and strong--- coincide on
$V$ considered as a subspace of $QM(V)$. This assertion may be
strengthened in the following way.

\begin{lem}\label{lem:qs_and_strong_top}
 $\nu_{a,b}(q)=\mu_{a,b}(q)$, i.e. $\|q(a,b)\|= \|a\lhd q\rhd b\|$,
for any $q\in QM(V)$, $a\in A, b\in B$.
\end{lem}

\begin{proof}
Let $q \in QM(V)$, $a\in A, b\in B$ be given, choose approximate
units $\{e_{\alpha}\}$ in $A$ and $\{u_{\beta}\}$ in $B$. Then the
net $q(e_{\alpha}a,bu_{\beta})=(a\lhd q\rhd
b)(e_{\alpha},u_{\beta})$ converges in norm to $q(a,b)$. It
implies that $\|q(a,b)\|=\lim\|a\lhd q\rhd
b(e_\alpha,u_\beta)\|\le \|a\lhd q\rhd b\|$. On the other hand,
\begin{align*}
\|a\lhd q\rhd b\|&=\sup\{\|(a\lhd q\rhd b)(c,d)\| :
\|c\|\le 1, \|d\|\le 1, c\in A, d\in B\}\\
&=\sup\{\|q(ca,bd)\| : \|c\|\le 1, \|d\|\le 1, c\in A, d\in B\}\\
&=\sup\{\|cq(a,b)d\| : \|c\|\le 1, \|d\|\le 1, c\in A, d\in B\}\\
&\le \|q(a,b)\|,
\end{align*}
which proves the inverse inequality.
\end{proof}

Consider the canonical embedding $\Gamma \colon  V \rightarrow
QM(V)$ given by (\ref{eq:V_into_QM(V)}).  This way $QM(V)$
provides an extension of $V$.

\begin{dfn}\label{dfn:quasi_extension}
{\rm
A \emph{quasi extension} of a Hilbert $A$-$B$-bimodule $V$
consists of:
\begin{enumerate}[(i)]
 \item two Banach algebras $\mathcal{A}$ and $\mathcal{B}$, such that $A \subset \mathcal{A}$
       is a right ideal and $B \subset \mathcal{B}$ is a left ideal,
 \item a Banach $\mathcal{A}$-$\mathcal{B}$-bimodule $W$
 \item and an isometric bimodule homomorphism $\Phi \colon  V \longrightarrow W$ with
       \[
     \text{Im}(\Phi) = A\,W B:=\overline{{\rm span}\{axb : a\in A, x\in W, b\in B\}}\ .
       \]
\end{enumerate}}
\end{dfn}

\begin{dfn}{\rm A quasi extension
$(W,\mathcal{A},\mathcal{B},\Phi)$ of $V$ is said to be
\emph{strictly essential} if $A\subset \mathcal{A}$ is a
\emph{right strictly essential ideal}, i.e.
\begin{gather}\label{eq:s_e_right_ext}
\|\alpha\|=\sup\{\|a\alpha\| : a\in A, \|a\|\le 1\}\quad \text{
for all } \alpha\in \mathcal{A},
\end{gather}
$B\subset \mathcal{B}$ is a \emph{left strictly essential ideal},
i.e.
\[
\|\beta\|=\sup\{\|\beta b\| : b\in B, \|b\|\le 1\}\quad \text{ for
all } \beta\in \mathcal{B}
\]
and the following condition holds
\begin{gather}\label{eq:s_e_quasi_ext}
    \|y\|=\sup\{\|ayb\|: a\in A, b\in B, \|a\|\le 1, \|b\|\le 1\} \quad \text{ for
all } y\in W.
\end{gather}
}
\end{dfn}

\begin{dfn}\rm\label{dfn:max_s_e_e}
A strictly essential quasi extension
$(\widehat{W},\widehat{\mathcal{A}},\widehat{\mathcal{B}},\widehat{\Phi})$
of $V$ is said to be \emph{maximal} if for any other strictly
essential quasi extension $(W,\mathcal{A},\mathcal{B},\Phi)$ there
are an isometric homomorphism $\lambda \colon
\mathcal{A}\rightarrow \widehat{\mathcal{A}}$, which is the
identity on $A$, an isometric homomorphism $\mu \colon
\mathcal{B}\rightarrow \widehat{\mathcal{B}}$, which is the
identity on $B$ and an isometric linear map $\Theta \colon
W\rightarrow \widehat{W}$ such that it satisfies the condition
\begin{gather}\label{eq:s_e_ext_mod_cond}
    \Theta (ayb)=\lambda(a)\Theta(y)\mu(b)\quad\text{for all}\quad
    a\in \mathcal{A}, y\in W, b\in \mathcal{B},
\end{gather}
and such that the diagram
\[
\begin{diagram}
\node{W}\arrow[2]{e,t}{\Theta} \node[2]{\widehat{W}}
\\ \node[2]{V}\arrow{ne,r}{\widehat{\Phi}}\arrow{nw,b}{\Phi}
\end{diagram}
\]
is commutative.
\end{dfn}

\begin{teo} Given an $A$-$B$-bimodule $V$. Then
$(QM(V),RM(A),LM(B),\Gamma)$ is a maximal strictly essential quasi
extension of $V$, where $\Gamma$ is defined by
(\ref{eq:V_into_QM(V)}).
\end{teo}

\begin{proof} $A\subset RM(A)$ is a right strictly essential ideal
and $B\subset LM(B)$ is a left strictly essential ideal by
\cite[Lemma 6]{Pav1}. Using approximate units of $A$ and
$B$ a straightforward verification yields the formula
(\ref{eq:s_e_quasi_ext}). Now let us check the third condition of
Definition \ref{dfn:quasi_extension}. Obviously, $\text{Im}
\Gamma\subset A QM(V) B$ and we only have to ensure the inverse
inclusion. Given arbitrary $q\in QM(V), a\in A, b\in B$. Then for
any $c\in A, d\in B$ one has
\[
(a\lhd q\rhd b)(c,d)=q(ca,bd)=cq(a,b)d=\Gamma(q(a,b))(c,d).
\]
Because $\Gamma$ is an isometry, $\text{Im}(\Gamma)$ is closed,
hence $$\text{Im} \Gamma= A QM(V) B$$ and
$(QM(V),RM(A),LM(B),\Gamma)$ is a  strictly essential quasi
extension of $V$. To establish its maximality one chooses any
other strictly essential quasi extension
$(W,\mathcal{A},\mathcal{B},\Phi)$ of $V$. By \cite{Pav1} $LM(B)$
is a maximal left strictly essential extension of $B$ and,
consequently, there is an isometric homomorphism $\mu \colon
\mathcal{B}\rightarrow LM(B)$, which restricts to the identity on
$B$. Similarly, there is an isometric homomorphism $\lambda \colon
\mathcal{A}\rightarrow RM(A)$, which acts identically on $A$. Now
for $y\in W, a\in A, b\in B$ put
\[
\Xi (y)(a,b)=\Phi^{-1}(ayb).
\]
Obviously, $\Xi (y)$ is a bilinear map from $A\times B$ to $V$.
Moreover, $\Xi$ is actually an isometry, because
\begin{align*}
    \|\Xi(y)\|&=\sup\{\|\Phi^{-1}(ayb)\| : a\in A, b\in B,
\|a\|\le1, \|b\|\le 1\}\\
&=\sup\{\|ayb\| : a\in A, b\in B, \|a\|\le1, \|b\|\le
1\}=\|y\|,
\end{align*}
where we have used item (iii) of Definition
\ref{dfn:quasi_extension} and condition (\ref{eq:s_e_quasi_ext}).
Now choose $a\in A, \alpha\in \mathcal{A}, b\in B, \beta\in
\mathcal{B}$ and $y\in W$. On the one hand one has \[ \Xi(\alpha
y\beta)(a,b)=\Phi^{-1}(a\alpha y\beta b) \] and on the other hand
\begin{align*}
(\lambda(\alpha)\lhd \Xi(y)\rhd\mu(\beta))(a,b)
&=\Xi(y)(\lambda(\alpha)(a),\mu(\beta)(b))\\
&=\Phi^{-1}([\lambda(\alpha)(a)]y[\mu(\beta)(b)])\\
&=\Phi^{-1}(a\alpha y \beta b).
\end{align*}
So, the map $\Xi$ satisfies the condition
(\ref{eq:s_e_ext_mod_cond}). The theorem is proved.
\end{proof}

%%%%%%%%%%%%%%%%%%%%%%%%%%%%%%%%%%%%%%%%%%%%%%%%%%%%
\section{Quasi-multipliers of Hilbert $C^*$-bimodules in Kasparov's sense}
%%%%%%%%%%%%%%%%%%%%%%%%%%%%%%%%%%%%%%%%%%%%%%%%%%%%
Let us begin by recalling the definition of Hilbert
$C^*$-bimodules in Kasparov's sense, which is the starting point
for $KK$-theory (cf., e.g., \cite{Kasparov80b}). Given two
$C^*$-algebras $A$ and $B$, one considers a right
$\mathbb{Z}/2\mathbb{Z}$-graded Hilbert $B$-module $V$ and a
$*$-homomorphism $\rho\colon  A\rightarrow
\End^*_B(V)^{(0)}$, where $\End^*_B(V)^{(0)}$ denotes the
$0$-homogeneous adjointable operators  in $V$. We will
additionally assume that this representation is faithful and
non-degenerate. Then, in particular, the $C^*$-algebra $\rho (A)$
is isomorphic to $A$, and its left action on $V$ is given by the
formula
\[
a\lhd x=\rho(a)(x),\quad a\in A, x\in V.
\]
The right action of $B$ on $V$ will sometimes be denoted by
\[
x\rhd b=xb, \quad b\in B, x\in V.
\]
Let us remark that, in fact, we may restrict our considerations
concerning (left, right or quasi) multipliers of $V$ to the
non-graded case, because $\End^*_B(V)^{(0)}=\End^*_B(V_1)\oplus
\End^*_B(V_2)$, where $V=V_1\oplus V_2$ means the given
$\mathbb{Z}/2\mathbb{Z}$-graduation of $V$. So henceforth we
assume that the module $V$ is non-graded and the (faithful,
non-degenerate) representation $\rho$ is of the form $\rho\colon
A\rightarrow \End^*_B(V).$

\begin{dfn}{\rm
Quasi-multipliers $QM_{(A,\rho,B)}(V)$ of $V$ are defined as the
set of all bimodule homo\-mor\-phisms from $A\times B$ to $V$,
i.e.
\[
QM_{(A,\rho,B)}(V)={\rm Hom}_{A,B}(A\times B, V).
\] }
\end{dfn}

The Banach space of quasi-multipliers $QM_{(A,\rho,B)}(V)$ carries
an $RM(A)$-$LM(B)$-bimodule structure via
\[
 (q \rhd \beta)(a,b) = q(a, \beta(b)) \quad ,
 \quad (\alpha \lhd q)(a,b) = q(\alpha(a), b)
\]
for $\alpha \in RM(A)$, $\beta \in LM(B)$.

\begin{prop}
$V$ is isometrically embedded into $QM_{(A,\rho,B)}(V)$ by the
bimodule map
\begin{gather}\label{eq:V_into_QM_A_rho_B(V)}
\Gamma_{(A,\rho,B)} \colon  V\rightarrow QM_{(A,\rho,B)}(V),\quad
\Gamma_{(A,\rho,B)}(x)(a,b):=a\lhd x\rhd b:=\rho(a)(xb).
\end{gather}
\end{prop}

\begin{proof} Given $x\in V$, $a, a'\in A$, $b, b'\in B$. Denote the quasi-multiplier
$\Gamma_{(A,\rho,B)}(x)$ by $q_x$ for brevity. Then
\begin{align*}
q_{a'\lhd x \rhd b'}(a,b)%  &=
% \rho(a)\lhd (\rho(a')\lhd x \rhd b')\rhd b\\
&= \rho(a)(\rho(a')(xb')b)\\
&=q_x(\rho(a)\rho(a'),b'b)\\
&=(a'\lhd q_x\rhd b')(a,b),
\end{align*}
so $\Gamma_{(A,\rho,B)}$ is a bimodule map and it only remains to
check that it is an isometry. Then
\[
\|q_x\|=\sup\{\|\rho(a)(xb)\| : \|a\|\le 1, \|b\|\le 1, a\in A,
b\in B\}\le \|x\|
\]
and we have to show that this supremum achieves the value $\|x\|$. For
this it is enough to verify that
\begin{gather}\label{eq:KK_Gamma_isometry}
\|x\|=\sup\{\|\rho(a)(x)\| : \|a\|\le 1, a\in A\}.
\end{gather}
Because the representation $\rho$ is non-degenerate, the
sub-bimodule $W=\mathrm{span}\{\rho(a)(x) : a\in A, x\in V\}$ is
dense in $V$ and, consequently, we need to prove
(\ref{eq:KK_Gamma_isometry}) only for the vectors $x\in W$. So,
choose an arbitrary $x\in W$, i.e. $x=\sum \rho(a_i)y_i$ with
$y_i\in V$. Let $\{e_\alpha\}$ be an approximate unit of $A$. Then
$\rho(e_\alpha)x=\sum \rho(e_\alpha a_i)y_i$ converges to $x$, and
the supremum in (\ref{eq:KK_Gamma_isometry}) achieves the norm
$\|x\|$ on the approximate unit $\{\rho(e_\alpha)\}$ of $\rho(A)$.
\end{proof}

In fact, we may carry out these considerations even for the
category of Banach bimodules over $C^*$-algebras, which act
non-degenerately. More precisely, given two $C^*$-algebras $A$ and
$B$ and a Banach $A$-$B$-bimodule $X$ such that the following
conditions hold
\begin{gather}\label{eq:dfn_qm_ban_bimod1}
\overline{{\rm span}\{ax : a\in A, x\in X\}}=X
\end{gather}
and
\begin{gather}\label{eq:dfn_qm_ban_bimod2}
\overline{{\rm span}\{xb : b\in B, x\in X\}}=X.
\end{gather}
Then quasi-multipliers $QM(X)$ of $X$ are defined again as the set
$\Hom_{A,B}(A\times B,X)$.

\begin{lem}\label{lem:eq_cond_ban_QM}
The two conditions (\ref{eq:dfn_qm_ban_bimod1}) and
(\ref{eq:dfn_qm_ban_bimod2}) are equivalent to the following one
\[ \overline{{\rm span}\{axb : a\in A, b\in B, x\in X\}}=X.\]
\end{lem}

\begin{proof} Let $X$ satisfy both (\ref{eq:dfn_qm_ban_bimod1}) and
(\ref{eq:dfn_qm_ban_bimod2}) and let an arbitrary $y\in X$ and
$\varepsilon>0$ be given. There are $a_i\in A$, $x_i\in X$ such
that
\[
\biggl\|y-\sum_{i=1}^n a_ix_i\biggr\|<\varepsilon
\]
and for any $i$ there are $b_j^{(i)}\in B$, $z_j^{(i)}\in X$ such
that
\[
\biggl\|x_i-\sum_{j=1}^{m_i}
z_j^{(i)}b_j^{(i)}\biggr\|<\frac{\varepsilon}{\|a_i\|n}\,.
\]
Then
\[
\biggl\|y-\sum_{i=1}^n\sum_{j=1}^{m_i}
a_iz_j^{(i)}b_j^{(i)}\biggr\| \le \biggl\|y-\sum_{i=1}^n
a_ix_i\biggr\|+\biggl\|\sum_{i=1}^n
a_ix_i-\sum_{i=1}^n\sum_{j=1}^{m_i}a_iz_j^{(i)}b_j^{(i)}\biggr\|<2\varepsilon.
\]
The inverse implication of the lemma is trivial.
\end{proof}

\begin{prop}
$X$ is isometrically embedded into $QM(X)$ by the bimodule map
\[
\Gamma \colon  X\rightarrow QM(X),\quad  \Gamma(x)(a,b)=axb.
\]
\end{prop}

\begin{proof} We only have to check that for any $x\in X$ one has
\[
\|x\|=\sup\{\|axb\| : \|a\|\le 1, \|b\|\le 1, a\in A, b\in B\}.
\]
By Lemma \ref{lem:eq_cond_ban_QM} the vector $x$ may be
approximated in norm by vectors of the form $\sum c_iy_id_i$ with
$c_i\in A$, $y_i\in X$, $d_i\in B$. Then
\[
\|\sum c_iy_id_i\|=\sup\{\|e_\alpha \sum c_iy_id_iu_\beta\| : \alpha, \beta\},
\]
where $\{e_\alpha\}$ and $\{u_\beta\}$ stand for approximate units
in $A$ and $B$ respectively.
\end{proof}

%%%%%%%%%%%%%%%%%%%%%%%%%%%%%%%%%%%%%%%%%%%%%%%%%%%%
\section{Quasi-multipliers of direct sums of bimodules}
%%%%%%%%%%%%%%%%%%%%%%%%%%%%%%%%%%%%%%%%%%%%%%%%%%%%
Given two $C^*$-algebras $A$ and $B$ and a Hilbert
$A$-$B$-bimodule $V$. Consider another $A$-$B$-bimodule
$\widetilde{V}$ and a bimodule homomorphism $\theta\colon  V\rightarrow
\widetilde{V}$. Then there is a homomorphism $\theta_*\colon
QM(V)\rightarrow QM(\widetilde{V})$ of Banach
$RM(A)$-$LM(B)$-bimodules given by the formula
\[
\theta_*(q)=\theta q,\quad q\in QM(V).
\]
So, quasi-multipliers provide a covariant functor $QM$ from the
category of Hilbert $A$-$B$-bimodules to the category of Banach
$RM(A)$-$LM(B)$-bimodules. Obviously, these observations are still
valid for Banach (instead of Hilbert) $A$-$B$-bimodules, on which
both $C^*$-algebras $A$ and $B$ act non-degenerately. If $V$ is
given as a direct sum $V=V_1\oplus V_2$ of its (closed)
sub-bimodules $V_1$ and $V_2$, then one straightforwardly verifies
that $QM(V)=QM(V_1)\oplus QM(V_2)$, in other words the functor
$QM$ is additive. In particular, for the free $A$-$A$-bimodule
$A^n$ one has $QM(A^n)=QM(A)^n$.

Now we are investigating what happens with quasi-multipliers if we
map either $A$ or $B$ to other $C^*$-algebras. So, consider two
$C^*$-algebras $\widetilde{A}$ and $\widetilde{B}$ and two
surjective $*$-homomorphisms
\[
\varphi\colon  A\rightarrow \widetilde{A},\quad \psi\colon  B\rightarrow
\widetilde{B}.
\]
Assume $V$ is a Banach $\widetilde{A}$-$\widetilde{B}$-bimodule
equipped with non-degenerate actions of these $C^*$-algebras.
Define a left action $\lhd_\varphi$ of $A$ twisted by $\varphi$
and right action $\rhd_\psi$ of $B$ twisted by $\psi$ on $V$ as
follows
\[
a\lhd_\varphi x=\varphi(a)\lhd x,\quad x\rhd_\psi b=x\rhd \psi(b),
\]
where $a\in A, b\in B$, $x\in V$. Surjectivity of $\varphi$ and
$\psi$ ensures that these actions are non-degenerate. Then
$(V,\lhd_\varphi,\rhd_\psi)$ is a Banach $A$-$B$-bimodule and
quasi-multipliers of this bimodule are called \emph{twisted
quasi-multipliers} of the original $\tilde A$-$\tilde B$-bimodule
$(V,\lhd,\rhd)$ and are denoted by $QM_{(\varphi,\psi)}(V)$.

With this construction, quasi-multipliers are contravariant in
both variables $A$ and $B$.

Now we are going to study the behavior of the functor $QM$ with
respect to infinite direct sums of bimodules. As a corollary, in
particular, we will obtain a description of quasi-multipliers for
the standard $A$-$A$-bimodule $l_2(A)$. So given $A$-$B$-bimodules
$V_i$. Obviously, for a sequence $(x_i), x_i\in V_i$ the series
$\sum_i {_A}\langle x_i,x_i\rangle$ converges in norm if and only
if the series $\sum_i \langle x_i,x_i\rangle_B$ does, moreover,
their norms have to coincide. Set
\[
V=\Bigl\{(x_i) : x_i\in V_i, \sum_i {_A}\langle x_i,x_i\rangle
 \text{ converges in norm}\Bigr\}.
\]
Then $V$ is a Hilbert $A$-$B$-bimodule with respect to the inner
products
\[
{_A}\langle x,y\rangle=\sum_i {_A}\langle
x_i,y_i\rangle\quad\text{and}\quad \langle x,y\rangle_B=\sum_i
\langle x_i,y_i\rangle_B,
\]
where $x=(x_i), y=(y_i)\in V$ (cf. \cite[Example
1.3.5]{MaTroBook}).

\begin{teo} Set
\begin{align*}
W=\Bigl\{(q_i) & :   q_i\in QM(V_i), \text{ the norms of the
operators }\\& \rho_n=(q_1,\dots,q_n,0,\dots)\colon  A\times
B\rightarrow \oplus_{i=1}^n V_i \text{ are uniformly bounded over } n,\\
&\text{and } (q_i(a,b))\in V \text{ for any } a\in A, b\in
B\Bigr\}.
\end{align*}
In particular, if $(q_i)\in W$ then both series  $\sum_i
{_A}\langle q_i(a,b),q_i(a,b)\rangle$ and $\sum_i \langle
q_i(a,b),q_i(a,b)\rangle_B$ converge in norm for any $a\in A, b\in
B$. Then $W$, with norm defined by
  \eqref{eq:def_norm} below, is a Banach
$RM(A)$-$LM(B)$-bimodule with entry-wise action, isometrically
isomorphic to the Banach $RM(A)$-$LM(B)$-bimodule $QM(V)$.
\end{teo}

\begin{proof} Suppose $r\in RM(A)$, $l\in LM(B)$ and $q=(q_i)\in
W$. Then
\[
\sum_i {_A}\langle (r\lhd q_i\rhd l)(a,b),(r\lhd q_i\rhd
l)(a,b)\rangle=\sum_i {_A}\langle
q_i(r(a),l(b)),q_i(r(a),l(b))\rangle
\]
and $r\lhd q \rhd l:=(r\lhd q_i \rhd l)$ belongs to $W$. Set
\[
\|q(a,b)\|:=\|\sum_i {_A}\langle q_i(a,b),q_i(a,b)\rangle\|^{1/2}
\]
and
\begin{equation} \|q\|:=\sup\{\|q(a,b)\| : \|a\|\le 1,
\|b\|\le 1, a\in A, b\in B\}.\label{eq:def_norm}
\end{equation}
This supremum is finite, because $q$ is a point-wise limit of the
sequence
\[
\{\rho_n=(q_1,\dots,q_n,0,\dots)\}
\]
 and $\|\rho_n\|\le C$ for
any $n$. Thus, $W$ is a normed $RM(A)$-$LM(B)$-bimodule. Note,
moreover, that $q$ considered as a map
  $q\colon A\times B\to V$ is bounded and thus a quasi-multiplier.

An isometric isomorphism $\Phi \colon  QM(V)\rightarrow W$ may be
defined in the following way. Denote by $p_i\colon  V\rightarrow
V_i$ the natural projection and consider any quasi-multiplier
$T\in QM(V)$, i.e. $T\colon  A\times B\rightarrow V$. Then,
clearly, $T_i=p_iT$ belongs to $QM(V_i)$ for any $i$ and the
sequence $\{F_n=(T_1,\dots,T_n,0,\dots)\}\subset QM(V)$
quasi-strictly converges to $T$. By definition set $\Phi
(T)=(T_i)$.

Because $T(a,b)=(T_1(a,b),T_2(a,b),\dots)\in \oplus V_i$ for any
$a\in A, b\in B$, the sequence $(T_i)$ belongs to $W$. Obviously,
$\Phi$ is an isometry. Now take an arbitrary $(q_i)\in W$. Define
$T(a,b):=(q_1(a,b),q_2(a,b),\dots)$ for $a\in A, b\in B$. Then $T$
is an element of $QM(V)$ and $\Phi(T)=(q_i)$, proving surjectivity
of $\Phi$.
\end{proof}

\begin{cor} Quasi-multipliers of the standard bimodule $l_2(A)$
over a $C^*$-algebra $A$ coincide with the set of sequences
$\{(q_i),\, q_i\in QM(A)\}$ such that the norms of
$\{\oplus_{i=1}^n q_i\}$ are uniformly bounded over $n$ and
$\sum_i (aq_ic)^*(aq_ic)$ converges in norm for any $a,c\in A$
$\Box$
\end{cor}

Let $V$ be a right Hilbert module over a $C^*$-algebra $B$. Then
its multipliers were defined in \cite{Bakic, RT02} as
$\Hom^*_B(B,V)$. It is a Hilbert module over the $C^*$-algebra
$M(B)$. Likewise, the left multipliers of $V$ were defined in
\cite{FraPav07} as $\Hom_B(B,V)$ being a Banach module over the
Banach algebra $LM(B)$. The arguments above imply the following
assertion.

\begin{teo} Assume that $V=\oplus V_i$ is a direct sum of Hilbert
    $B$-modules
$V_i$. Then
\begin{align*}
    LM(V)=\Bigl\{(\lambda_i) &: \lambda_i\in LM(V_i), \text{ the norms of the
operators }\\& \theta_n=(\lambda_1,\dots,\lambda_n,0,\dots)\colon
B\rightarrow \oplus_{i=1}^n V_i \text{ are uniformly bounded over } n,\\
    &\text{and } (\lambda_i(b))\in V \text{ for any } b\in B\Bigr\},\\
    M(V)=\Bigl\{(\mu_i) &: \mu_i\in M(V_i), \text{ the norms of the
operators }\\&\tau_n=(\mu_1,\dots,\mu_n,0,\dots)\colon  B\rightarrow
\oplus_{i=1}^n V_i
 \text{ are uniformly bounded over } n,\\ & \text{and
}(\mu_i(b))\in V \text{ for any } b\in B\Bigr\}.\quad\Box
\end{align*}
\end{teo}

This theorem in its part concerning multipliers generalizes
\cite[Theorem 2.1]{Bakic}, where the crucial case of the standard
module was considered. Our description being applied to $V=l_2(A)$
differs from the one of \cite{Bakic}, but is just its equivalent
reformulation. Indeed, let $V=l_2(A)$, $m_i\in M(A)$ and the
sequence $\{\tau_n=(m_1,\dots,m_n,0,\dots)\}$ be given. Then one
has
\begin{align}\label{eq:mult_descr_comp}
    \|\tau_n\|^2&=\sup\{\|\langle \tau_n(a), \tau_n(a)\rangle\| : a\in A, \|a\|\le
    1\}\notag\\&=\sup\Bigl\{\Bigl\|\sum_{i=1}^n m_i(a)^*m_i(a)\Bigr\|: a\in A, \|a\|\le
    1\Bigr\}\\&=\sup\Bigl\{\Bigl\|\sum_{i=1}^n a^*m_i^*m_ia\Bigr\|: a\in A, \|a\|\le
    1\Bigr\}\notag\\&=\Bigl\|\sum_{i=1}^n m_i^*m_i\Bigr\|.\notag
\end{align}
Now, \cite[Theorem 2.1]{Bakic} claims that
\[
M(l_2(A))=\Bigl\{(m_n) : m_n\in M(A), \sum am_n^*m_n, \sum
m_n^*m_na \text{ converge in } A \text{ for any } a\in A\Bigr\}.
\]
But the norm-convergence of a series $\sum a^*m_n^*m_na$ and
uniform boundedness of the sequence $\{\|\sum m_n^*m_n\|\}$ (say
by a constant $C$), which is ensured by the equality
(\ref{eq:mult_descr_comp}), imply the norm convergence of the
series $\sum ax_n^*x_n$ and $\sum x_n^*x_na$ because of the
Cauchy-Schwartz inequality
\[
\Bigl\|\sum m_n^*m_na\Bigr\|\le \Bigl\|\sum m_n^*m_n\Bigr\|^{1/2}
\cdot \Bigl\|\sum a^*m_n^*m_na\Bigr\|^{1/2} \le \Bigl\|\sum
a^*m_n^*m_na\Bigr\|^{1/2}C^{1/2}.
\]

%%%%%%%%%%%%%%%%%%%%%%%%%%%%%%%%%%%%%%%%%%%%%%%%%%%%
\section{Quasi-multipliers of continuous sections
of Hilbert $C^*$-bimodule bundles}
%%%%%%%%%%%%%%%%%%%%%%%%%%%%%%%%%%%%%%%%%%%%%%%%%%%%
Given a locally compact Hausdorff space $X$. For the commutative
$C^*$-algebra $C_0(X)$ of continuous functions on $X$ vanishing at
infinity its set of multipliers (as well as its set of left (or
right) multipliers and quasi-multipliers) coincides with the
$C^*$-algebra $C_b(X)$ of bounded continuous functions on $X$. On
the other hand $C_b(X)$ is nothing else but the $C^*$-algebra
$C(\beta X)$ of continuous functions on the Stone-\v{C}ech
compactification of $X$ (cf. \cite[3.12.6]{Pedersen}). This result
was extended in \cite{APT73} to $C^*$-algebras $A_0(X)=C_0(X,A)$
of continuous $A$-valued functions vanishing at infinity, where
$A$ is a $C^*$-algebra (actually in \cite{APT73} there was
considered the even  more general case of continuous cross
sections of fiber spaces). Denote by $M(A)_\beta$ the
$C^*$-algebra of multipliers of $A$, equipped with the strict
topology, and by $C_b(X,M(A)_\beta)$ the set of continuous bounded
$M(A)$-valued functions on $X$. Then
\begin{gather}\label{eq:APT73}
M(A_0(X))=C_b(X,M(A)_\beta)
\end{gather}
(cf.~\cite[Corollary 3.4]{APT73}). But $C_b(X,M(A)_\beta)$
\emph{is not} isomorphic to $C(\beta X, M(A)_\beta)$, because
$C(\beta X) \otimes M(A) = M(C_0(X)) \otimes M(A) \subsetneqq
M(C_0(X) \otimes A) = M(A_0(X))$ whenever $X$ is
$\sigma$-compact, $A$ is infinite dimensional
and the tensor products are considered with respect to the
minimal (spatial) norm, \cite[Theorem 3.8]{APT73}.

And in turn formula (\ref{eq:APT73}) was extended in
\cite{Frank85} in the following way. Let $V$ be a Hilbert
$A$-module and $V_0(X)=C_0(X,V)$ be the set of continuous
$V$-valued functions vanishing at infinity. It is, obviously, a
Hilbert $A_0(X)$-module. Denote by $\End^*_A(V)_\beta$ the
$C^*$-algebra of all $A$-linear bounded adjointable operators in
$V$, equipped with the $*$-strict module topology (cf. \cite[\S
5.5]{MaTroBook}). Then
\begin{gather}\label{eq:Fra85}
\End^*_{A_0(X)}(V_0(X))=C_b(X,\End^*_A(V)_\beta).
\end{gather}
Because by Kasparov's theorem $\End^*_A(V)=M(K_A(V))$ (cf.
\cite{Kasparov80}) for any Hilbert $A$-module $V$, where $K_A(V)$
stands for the $C^*$-algebra of compact operators of $V$, the
formula (\ref{eq:APT73}) is a particular case of (\ref{eq:Fra85})
for $V=A$. Our aim in this paragraph is to find the proper
analogue of formula (\ref{eq:APT73}) for quasi-multipliers of
continuous sections of Hilbert $C^*$-bimodule bundles.

In order to define this notion, take a locally compact Hausdorff
space $X$ and two $C^*$-algebras $A$ and $B$, set $A_0(X) :=
C_0(X,A)$ and $B_0(X) := C_0(X,B)$. Equipped with the supremum
norm, these are again $C^*$-algebras.

In view of the above observations we want sections in our still to
be defined bundles of Hilbert $A$-$B$-bimodules to form a Hilbert
$A_0(X)$-$B_0(X)$-bimodule with the inner product induced by the
pointwise operations in the fibers. The corresponding structure
group should therefore reduce to unitary $A$-$B$-linear operators,
which raises the question whether these are well-defined, since we
have two
 inner products. This is settled by the following lemma.
\begin{lem}
Let $V$ be a Hilbert $A$-$B$-bimodule and $T \in \End_{A,B}(V)$
be a bounded $A$- and $B$-linear operator, which has an adjoint
$T^{*,B}$ for the $B$-valued inner product. Then $T^{*,B}$
coincides with the adjoint of $T$ for the $A$-valued inner
product (i.e. $T^{*,A} = T^{*,B}$).
\end{lem}
\begin{proof}
We follow \cite[Remark 1.9]{BrMinShen}. Let $x,y,z \in V$, then we
have
\begin{eqnarray*}
_A\langle x,\,T^{*,B}\,y \rangle\,z
&=& x\,\langle T^{*,B}\,y,\,z\rangle_B
= x\,\langle y,\,T\,z\rangle_B
= \,_A\langle x,\, y \rangle \,T\,z
= T\left(_A\langle x,\,y \rangle\,z\right) \\
&=& T\left(x\,\langle y,\,z \rangle_B\right) = Tx\,\langle y,\,z
\rangle_B = \,_A\langle Tx,\,y \rangle\,z
\end{eqnarray*}
Clearly $a = \,_A\langle x,\,T^{*,B}\,y \rangle - \,_A\langle
Tx,\,y \rangle \in \,_A\langle V, V \rangle$, where the latter
denotes the closure of the linear span of all possible $A$-valued
inner products. Moreover $az = 0$ for all $z \in V$ by the
previous calculation. This implies $a = 0$ by the approximate unit
argument given in \cite[Remark 1.9]{BrMinShen}.
\end{proof}

\begin{dfn}
Let $V$ be a Hilbert $A$-$B$-bimodule. By the above lemma, the
adjointable $A$-$B$-linear operators $\End^*_{A,B}(V)$ are
well-defined. Denote the unitary elements in this $C^*$-algebra
by $U_{A,B}(V)$.
\end{dfn}

\begin{dfn}
Given a locally compact Hausdorff space $X$ and a Hilbert
$A$-$B$-bimodule $V$. A Hilbert $A$-$B$-bimodule bundle
$\mathcal{V}$ over $X$ with typical fiber $V$ is a
triple $(\mathcal{V},p,X)$, where $\mathcal{V}$ is a
Hausdorff space and $p \colon \mathcal{V} \rightarrow X$
maps $\mathcal{V}$ onto $X$ such that the following holds:
\begin{enumerate}[(i)]
\item \label{it:fib} there is an open cover
$\{U_i\}_{i \in I}$ of $X$ such that there exist
homeomorphisms
\[
\varphi_i \colon p^{-1}(U_i) \longrightarrow U_i \times V
\]
with $\text{\rm pr}_1 \circ \varphi_i = \left.p\right|_{p^{-1}(U_i)}$.
\item \label{it:unitary} let $\overline{\varphi}_{ij}$ be defined
via $\varphi_j \circ \varphi_i^{-1}(x,v) = (x,
\overline{\varphi}_{ij}(x)(v))$ for $x \in U_i \cap U_j$
and $v \in V$, then $\overline{\varphi}_{ij}$ is a
continuous map
\[
\overline{\varphi}_{ij} \colon U_i \cap U_j \longrightarrow
U_{A,B}(V)\ .
\]
\end{enumerate}
\end{dfn}

Condition (\ref{it:fib}) implies that $\mathcal{V}$ is fiberwise
isomorphic to $V$, condition (\ref{it:unitary}) encodes the
reduction of the structure group to the unitary operators. The
continuous sections $\mathcal{V}_0(X) = C_0(X, \mathcal{V})$
indeed yield a $A_0(X)$-$B_0(X)$-bimodule. Let $x \in X$ be in the
set $U_i$ of the cover, then  there is an $A_0(X)$-valued
inner product on $\mathcal{V}_0(X)$ defined via
\[
_{A_0(X)}\langle \sigma,\,\tau \rangle (x) =\,_A\langle
\mathrm{pr}_2\circ\varphi_i \circ \sigma
(x),\,\mathrm{pr}_2\circ\varphi_i \circ \tau (x)\rangle,
\]
where $\mathrm{pr}_2$ stands for the projection of $U_i\times V$
onto $V$. This does not depend on the particular choice of $(U_i,
\varphi_i)$. Indeed, if $x$ lies in $U_i \cap U_j$ we have:
\begin{align*}
\,_A\langle \mathrm{pr}_2\circ\varphi_i \circ \sigma
(x),\,\mathrm{pr}_2\circ\varphi_i \circ \tau (x)\rangle &=
\,_A\langle\,\overline{\varphi}_{ji}(x)(pr_2\circ\varphi_j \circ
\sigma (x)),\,\overline{\varphi}_{ji}(x)(pr_2\circ\varphi_j \circ
\tau (x))\rangle
 \\& =\,_A\langle \mathrm{pr}_2\circ\varphi_j \circ \sigma
(x),\,\mathrm{pr}_2\circ\varphi_j \circ \tau (x)\rangle
\end{align*}
due to the unitarity of the structure group. There is a similar
$B_0(X)$-valued inner product on $\mathcal{V}_0(X)$. With these
additional structures $\mathcal{V}_0(X)$ is indeed an
$A_0(X)$-$B_0(X)$-bimodule.

Associated to $\mathcal{V}$ we have the bundle of
quasi-multipliers $QM(\mathcal{V})$. To define this, note that for
a unitary $u \in U_{A,B}(V)$ and a quasi-multiplier $q \in QM(V)$
the map $u \circ q$ is again a quasi-multiplier due to the
$A$-$B$-linearity of $u$. Therefore the space
\[
\coprod_{i \in I} U_i \times QM(V)
\]
may be equipped with the equivalence relation
\[
(x,q) \sim (x,\overline{\varphi}_{ij}(x)\circ q),
\]
where $i,j \in I, x \in U_{ij}$ and $q \in QM(V)$. The quotient
$QM(\mathcal{V}) = \coprod_{i \in I} U_i \times QM(V) / \sim$ is a
locally trivial bundle with typical fiber $QM(V)$. Moreover the
canonical map $\iota \colon V \rightarrow QM(V)$ extends to a
bundle morphism
\[
\mathcal{V} \longrightarrow QM(\mathcal{V}) \; ; \quad v \mapsto
[x, \iota \circ \mathrm{pr_2}\circ \varphi_i(v)]\ ,
\]
where $v$ belongs to the fiber over $x\in X$ and $[x,q]\in
QM(\mathcal{V})$ denotes the equivalence class of $(x,q)$. We may
consider the quasi-strict topology on $QM(V)$, the quotient
topology induced by this on the space $QM(\mathcal{V})$ will again
be called the quasi-strict topology on the bundle
$QM(\mathcal{V})$. This is the last ingredient to phrase the
analogue of (\ref{eq:APT73}) in the case of bundles.

\begin{teo}\label{teo:q_m_comm_case}
For the quasi-multipliers of $\mathcal{V}_0(X)$ we have an
isometric bimodule isomorphism
\[
 QM(\mathcal{V}_0(X)) \cong C_b(X, QM(\mathcal{V}))\ ,
\]
where $QM(\mathcal{V})$ on the right-hand side is equipped with the
quasi-strict topology.
\end{teo}

\begin{proof}
We are going to construct explicit maps in both directions and
show that they are inverse to each other. Denote by $\pi\colon
QM(\mathcal{V}) \rightarrow X$ the bundle projection. For the map
from the left hand side to the right we need an evaluation map
turning a quasi-multiplier on sections $QM(\mathcal{V}_0(X))$ into
a quasi-multiplier on a fixed fiber $QM(\mathcal{V})_y =
\pi^{-1}(y)$. Therefore we need to be able to construct sections
of $A_0(X)$, $B_0(X)$ with a prescribed value at a given point $y
\in X$. Local compactness enables us to achieve this. Let $a \in
A$, $b \in B$ be given. By passing to the one-point
compactification $X^+$ (which is normal) we can construct a
function
\[
 \chi^y \colon  X^+ \longrightarrow [0,1]
\]
which is $1$ at $y$ and vanishes at $\infty$. In particular, we
may set $\alpha = \chi^y\,a \in A_0(X)$ and $\beta = \chi^y\,b
\in B_0(X)$.

If $\mathcal{V}_y$ denotes the fiber of $\mathcal{V}$ over $y \in X$,
then $QM(\mathcal{V})_y$ is by construction canonically isomorphic to
$QM(\mathcal{V}_y)$. Let $\alpha, \beta$ be sections of $A_0(X)$,
$B_0(X)$ as above and set
\[
 \varphi_y \colon QM(\mathcal{V}_0(X)) \longrightarrow QM(\mathcal{V})_y
\quad ; \quad \varphi_y(G)(a,b) = G(\alpha, \beta)(y)\ .
\]
To see that this does not depend on the choice of $\alpha$ note
that $G(\cdot, \beta) \colon  A_0(X) \rightarrow \mathcal{V}_0(X)$
is left $A_0(X)$-linear and bounded for any $\beta \in B_0(X)$,
therefore
\[
_{A_0(X)}\langle\,G(\alpha, \beta)\,,\,G(\alpha, \beta)\,\rangle\,
\leq\,\lVert G(\cdot, \beta) \rVert^2\cdot
\,_{A_0(X)}\langle\,\alpha,\,\alpha\,\rangle \ .
\]
If $\alpha(y) = 0$ this implies $_{A_0(X)}\langle\,G(\alpha,
\beta)\,,\,G(\alpha, \beta)\,\rangle(y) = 0$. Thus, $\varphi_y$
does not depend on the choice of $\alpha$. The same argument shows
that different choices of $\beta$ will lead to the same map
$\varphi_y$. Furthermore
\begin{gather}\label{eq:G_is_bounded}
 \lVert \varphi_y(G)(a,b) \rVert = \lVert G(\alpha, \beta)(y) \rVert
\leq \lVert G \rVert\,\lVert \alpha \rVert\,\lVert \beta \rVert
= \lVert G \rVert\,\lVert a \rVert\,\lVert b \rVert
\end{gather}
proves that $\varphi_y(G)$ is bounded and therefore indeed defines
an element of $QM(\mathcal{V}_y) = QM(\mathcal{V})_y$. Note that
the upper bound can be chosen independently of $y \in X$.

Recall that a section $\sigma \colon X \rightarrow
QM(\mathcal{V})$ is continuous at $y \in Y$ if and only if there
exists a trivialization ${\psi_U} \colon
\left.QM(\mathcal{V})\right|_U \rightarrow U \times QM(V)$ such
that the map $\mathrm{pr_2}\circ\psi_U \circ \left.\sigma\right|_U
\colon U \rightarrow QM(V)$ is continuous. Let $\phi_U \colon
\left.\mathcal{V}\right|_U \rightarrow U \times V$ be a local
trivialization of $\mathcal{V}$. By construction of
$QM(\mathcal{V})$ there is a corresponding trivialization $\psi_U$
such that for $y \in U, q \in QM(\mathcal{V})_y =
QM(\mathcal{V}_y)$, $a \in A$ and $b \in B$ we have
\[
(\mathrm{pr_2}\circ\psi_U(q))(a,b) = \mathrm{pr_2}\circ
\phi_U(q(a,b))\ .
\]

Now let $\varepsilon > 0$. Since $G(\alpha,\beta) \in \mathcal{V}_0(X)$ is
continuous at $y$, we can find an open neighborhood $U \ni y$ and
a trivialization $\phi_U : \left.\mathcal{V}\right|_U \rightarrow
U \times V$, such that
\[
\lVert \mathrm{pr_2}\circ\phi_U(G(\alpha,\beta)(y)) -
\mathrm{pr_2}\circ\phi_U(G(\alpha,\beta)(y')) \rVert \leq
\varepsilon
\]
for all $y' \in U$.  In view of our above observation this
proves continuity of $y \mapsto \varphi_y(G)$ with respect to
the quasi-strict topology, since applying Lemma
\ref{lem:qs_and_strong_top} one has
\begin{align*}
 &\lVert a \lhd \left(\mathrm{pr_2}\circ {\psi}_U(\varphi_y(G)) -
\mathrm{pr_2}\circ {\psi}_U(\varphi_{y'}(G))\right) \rhd b
\rVert \\
 =&\lVert \mathrm{pr_2}\circ {\psi}_U(\varphi_y(G))(a,b) -
 \mathrm{pr_2}\circ {\psi}_U(\varphi_{y'}(G))(a,b) \rVert \\
 =&\lVert \mathrm{pr_2}\circ \phi_U(\varphi_y(G)(a,b)) -
 \mathrm{pr_2}\circ\phi_U(\varphi_{y'}(G)(a,b))
\rVert \\
=&\lVert \mathrm{pr_2}\circ\phi_U(G(\alpha,\beta)(y)) -
\mathrm{pr_2}\circ\phi_U(G(\alpha,\beta)(y'))
 \rVert \leq \varepsilon\ .
\end{align*}
By the independence of the bound in (\ref{eq:G_is_bounded}) the
section constructed above is also bounded. Therefore
\[
 S \colon  QM(\mathcal{V}_0(X)) \longrightarrow C_b(X, QM(\mathcal{V}))
\quad , \quad G \mapsto (y \mapsto \varphi_y(G))\ .
\]
is well-defined, linear and satisfies $\lVert S(G) \rVert \leq
\lVert G \rVert$. For the inverse direction consider
\[
\Phi \colon  C_b(X,QM(\mathcal{V}))\rightarrow QM(\mathcal{V}_0(X))
\; , \quad \Phi(F)(\alpha,\beta)(x):=F(x)(\alpha(x),\beta(x))\ .
\]
First, we have to check that the element $\Phi(F)(\alpha,\beta)$
belongs to $\mathcal{V}_0(X)$, i.\,e. that the function
\[
x\mapsto F(x)(\alpha(x),\beta(x))
\]
vanishes at infinity and is continuous. For any $\varepsilon>0$
there is a compact  $K\subset X$ such that
$\|\alpha(x)\|<\varepsilon$ and $\|\beta(x)\|<\varepsilon$ for
$x\in X\setminus K$. Then
\begin{align}\label{eq:Phi_is_bounded}
\|\Phi(F)(\alpha,\beta)(x)\|&=\|F(x)(\alpha(x),\beta(x))\|\\
&\le \|F(x)\|\|\alpha(x)\|\|\beta(x)\|\le \|F\|\varepsilon^2\notag
\end{align}
for $x\in X\setminus K$ proving that it vanishes at infinity.

For the verification of continuity let $\varepsilon>0$ and $x\in X$.
There is a neighborhood $U_1$ of $x$ such that
\begin{gather*}\label{eq:qm_comm_1}
    \|\alpha(x)-\alpha(y)\|<\varepsilon\quad\text{and}\quad
    \|\beta(x)-\beta(y)\|<\varepsilon\quad\text{whenever } y\in
    U_1.
\end{gather*}
 On the other hand by Lemma
\ref{lem:qs_and_strong_top} there is a neighborhood $U_2 \subset U_1$
of $x$ such that
\begin{eqnarray*}\label{eq:qm_comm_2}
&& \left\|\left(\left.\mathrm{pr_2}\circ {\psi}_{U_2} \circ
F\right|_{U_2}\right)(x)(\alpha(x),\beta(x))
-\left(\left.\mathrm{pr_2}\circ {\psi}_{U_2} \circ
F\right|_{U_2}\right)(y)(\alpha(x),\beta(x))\right\|
\\
&=& \|\mathrm{pr_2}\circ\phi_{U_2}(F(x)(\alpha(x),\beta(x)))-
\mathrm{pr_2}\circ\phi_{U_2}(F(y)(\alpha(x),\beta(x)))\| <
\varepsilon
\end{eqnarray*}
whenever $y\in U_2$. One has
\begin{eqnarray*}
&&\|\mathrm{pr_2}\circ\phi_{U_2}(\Phi(F)(\alpha,\beta)(x))-
\mathrm{pr_2}\circ\phi_{U_2}(\Phi(F)(\alpha,\beta)(y))\| \\
&=& \|\mathrm{pr_2}\circ\phi_{U_2}(F(x)(\alpha(x),\beta(x)))-
\mathrm{pr_2}\circ\phi_{U_2}(F(y)(\alpha(y),\beta(y)))\|\\
    &\le&
    \|\mathrm{pr_2}\circ\phi_{U_2}(F(x)(\alpha(x),\beta(x)))-
    \mathrm{pr_2}\circ\phi_{U_2}(F(y)(\alpha(x),\beta(x)))\|\\
    &+&\|\mathrm{pr_2}\circ\phi_{U_2}(F(y)(\alpha(x),\beta(x)))-
    \mathrm{pr_2}\circ\phi_{U_2}(F(y)(\alpha(y),\beta(x)))\|\\
    &+&\|\mathrm{pr_2}\circ\phi_{U_2}(F(y)(\alpha(y),\beta(x)))-
    \mathrm{pr_2}\circ\phi_{U_2}(F(y)(\alpha(y),\beta(y)))\|
    \\&\le& \varepsilon+\|F\|\|\beta\|\varepsilon+\|F\|\|\alpha\|\varepsilon
\end{eqnarray*}
for $y\in U_2$, which proves continuity of
$\Phi(F)(\alpha,\beta)$. Together with the norm estimates
(\ref{eq:Phi_is_bounded}), this completes the proof of
well-definedness of $\Phi$. Clearly, $\Phi$ is the inverse of $S$.
Moreover, the inequalities (\ref{eq:G_is_bounded}) and
(\ref{eq:Phi_is_bounded}) ensure that $S$ is an isometry.
\end{proof}

\begin{rk}
The evaluation map $\varphi_y$ used in the proof coincides
with the extension of
\[
\varphi_y \colon \mathcal{V}_0(X) \longrightarrow \mathcal{V}_y
\]
with respect to the quasi-strict topology.
\end{rk}

Let $\mathcal{V}$ be a bundle of right Hilbert $B$-modules for a
$C^*$-algebra $B$. By a similar construction as the one given
above there is a bundle $LM(\mathcal{V})$ of left multipliers and
a bundle $M(\mathcal{V})$ of double multipliers. The above
arguments may be used to prove the following analogue of Theorem
\ref{teo:q_m_comm_case} for left and (double) multipliers.

\begin{teo}
There are the following isometric $B$-module isomorphisms
\begin{gather*}
 LM(\mathcal{V}_0(X)) \cong C_b(X, LM(\mathcal{V}))\ ,\\
 M(\mathcal{V}_0(X)) \cong C_b(X, M(\mathcal{V}))\ ,
\end{gather*}
where $LM(\mathcal{V})$ (resp., $M(\mathcal{V})$) on the
right-hand side are equipped with the left strict (resp.,
strict) topology. $\Box$
\end{teo}

{\bf Acknowledgement:}
The presented work was partially supported by the
joint RFBR and Deutsche Forschungsgemeinschaft (DFG) project (RFBR grant
07-01-91555 / DFG project
``K-Theory, $C^*$-algebras, and Index theory''). AP was partially supported by
the RFBR (grant 08-01-00867).
The research was carried out during
the visit of the first author to the Georg-August-Universit\"at
  G\"{o}ttingen,
and he appreciates its hospitality a lot.

The authors are grateful to the referee for his/her valuable suggestions.

\end{document}